\documentclass[12pt,reqno]{amsart}

\usepackage{amscd,amsmath,amsfonts,amsthm,epsfig, graphics, latexsym, mathrsfs}
\usepackage{amssymb}
\usepackage[unicode]{hyperref}

\usepackage[alphabetic,initials]{amsrefs}

\theoremstyle{plain}

\newtheorem{theorem}{Theorem}[section]
\newtheorem{proposition}[theorem]{Proposition}
\newtheorem{lemma}[theorem]{Lemma}
\newtheorem{corollary}[theorem]{Corollary}

\theoremstyle{definition}

\newcommand{\appsection}[1]{\let\oldthesection\thesection
\renewcommand{\thesection}{Appendix \oldthesection}
\section{#1}\let\thesection\oldthesection}

\newtheorem{definition}[theorem]{Definition}

\newtheorem{say}[theorem]{}
\newtheorem{remarks}[theorem]{Remarks}

\theoremstyle{remark}

\newtheorem{claim}{Claim}
\newtheorem{remark}[theorem]{Remark}

\def\Z{{\mathbb{Z}}}

\def\gp{\mathfrak{p}}

\def\I{{\mathcal{I}}}

\def\O{{\mathcal{O}}}

\def\X{{\mathcal{X}}}

\DeclareMathOperator{\Bl}{Bl}
\DeclareMathOperator{\Cl}{Cl}

\DeclareMathOperator{\Cox}{Cox}

\DeclareMathOperator{\HH}{H}
\DeclareMathOperator{\Mov}{Mov}
\DeclareMathOperator{\NN}{N}
\DeclareMathOperator{\NE}{\overline{NE}}

\DeclareMathOperator{\Nef}{Nef}

\DeclareMathOperator{\Pic}{Pic}
\DeclareMathOperator{\Proj}{Proj}
\DeclareMathOperator{\bProj}{\bold{Proj}}
\DeclareMathOperator{\Spec}{Spec}

\DeclareMathOperator{\ch}{char}
\DeclareMathOperator{\mult}{mult}
\DeclareMathOperator{\pr}{pr}

\newcommand{\QED}{\ifhmode\unskip\nobreak\fi\quad {\rm Q.E.D.}} 

\newcommand{\bA}{\mathbb A}

\newcommand{\bN}{\mathbb N}
\newcommand{\bP}{\mathbb P}
\newcommand{\bQ}{\mathbb Q}
\newcommand{\bR}{\mathbb R}

\newcommand{\bZ}{\mathbb Z}

\newcommand{\sD}{\mathcal D}
\newcommand{\cE}{\mathcal E}
\newcommand{\cF}{\mathcal F}
\newcommand{\cG}{\mathcal G}

\newcommand{\cI}{\mathcal I}
\newcommand{\cK}{\mathcal K}
\newcommand{\cL}{\mathcal L}
\newcommand{\cM}{\mathcal M}

\newcommand{\cO}{\mathcal O}

\newcommand{\cX}{\mathcal X}

\newcommand{\al}{\alpha}
\newcommand{\be}{\beta}
\newcommand{\ga}{\gamma}

\newcommand{\ra}{\rightarrow}

\newcommand{\dra}{\dashrightarrow}

\newcommand{\oM}{\overline{M}}
\newcommand{\oLM}{\overline{LM}}
\newcommand{\tLM}{\widetilde{LM}}

\newcounter{et}[section]
\def\cooltag{\tag{\arabic{section}.\arabic{et}}\addtocounter{et}{1}}

\begin{document}
\bibliographystyle{amsplain}
\title{$\oM_{0,n}$ is not a Mori Dream Space}

\author{Ana-Maria Castravet}
\author{Jenia Tevelev}

\address{Ana-Maria Castravet: \sf Department of Mathematics, The Ohio State University, 100 Math Tower, 231 West 18th Avenue, Columbus, OH 43210-1174} 
\email{noni@alum.mit.edu}

\address{\vskip -.5cm Jenia Tevelev: \sf Department of Mathematics, 
University of Massachusetts at Amherst, Lederle Graduate Research Tower, Amherst, MA 01003-9305} 
\email{tevelev@math.umass.edu}


\begin{abstract}
Building on the work of Goto, Nishida and Watanabe on symbolic Rees algebras of monomial primes, we prove that the moduli space of stable rational curves with $n$ punctures is not a Mori Dream Space for $n>133$. This answers a question of Hu and Keel. 
\end{abstract}

\maketitle

\section{Introduction}

We work over an algebraically closed field $k$. 
It~was argued that $\oM_{0,n}$ should be a Mori Dream Space (MDS for short) because it is ``similar to a toric variety'' and toric varieties are basic examples of MDS.
We suggest an adjustment to this principle: $\oM_{0,n}$~is similar to the blow-up of a toric variety at the identity element of the torus. Specifically, we prove the following. 
For any toric variety $X$, we denote by $\Bl_eX$ the blow-up of $X$ at the identity element of the torus. Let $\oLM_n$ be the Losev--Manin space \cite{LM}. It is a smooth projective toric variety of dimension $n-3$.

\begin{theorem}\label{asdazxvsfvsfvsdaqwf}
There exists a small $\bQ$-factorial projective modification $\tLM_{n+1}$ of $\Bl_e\oLM_{n+1}$ and surjective morphisms
$$\tLM_{n+1}\to\oM_{0,n}\to\Bl_e\oLM_n.$$
In particular,
\begin{itemize}
\item If $\oM_{0,n}$ is a MDS then $\Bl_e\oLM_n$ is a MDS.
\item If $\Bl_e\oLM_{n+1}$ is a MDS then $\oM_{0,n}$ is a MDS.
\end{itemize}
\end{theorem}

Next we invoke a beautiful theorem of Goto, Nishida, and Watanabe: 

\begin{theorem}[\cite{GNW}]\label{GNW theorem}
If $(a, b, c) = (7m-3, 5m^2-2m, 8m-3)$,  with $m\ge 4$ and $3\nmid m$, then $\Bl_e\bP(a,b,c)$ is not a MDS when $\ch k=0$.
\end{theorem}
We show that
\begin{theorem}\label{asdasdaqwf}
Let $n=a+b+c+8$, where $a,b,c$ are positive coprime integers.
If $\Bl_e\oLM_n$ is a MDS then $\Bl_e\bP(a,b,c)$ is a MDS. 
\end{theorem}

It immediately follows from these results,  answering the question of Hu--Keel \cite[Question 3.2]{HK}, that:  
\begin{corollary}\label{main}
Assume $\ch k=0$. Then $\oM_{0,n}$ is not a Mori Dream Space for $n\ge 134$. 
\end{corollary}

Understanding the birational geometry of the moduli spaces $\oM_{g,n}$ of stable, 
$n$-pointed genus $g$ curves is a problem that has received a lot of attention from many authors. Interest in the effective cone originated in the work of Harris and Mumford \cite{HM}
who showed that $\oM_{g,n}$ is a variety of general type for large $g$.
Mumford also raised the question of describing the ample divisors, i.e., ~the nef cone. 
A long standing conjecture of Fulton and Faber provides a conjectural description,
which was reduced to the case of genus $0$ by Gibney, Keel, and Morrison \cite{GKM}.
This prompted Hu and Keel \cite{HK} to raise the question if $\oM_{0,n}$ is a Mori Dream Space. 
In positive genus, this is known to be typically false. For example,
Keel proved in \cite{Keel} that, in characteriztic zero, $\oM_{g, n}$ is not a MDS for $g\geq3$, $n\geq1$, by proving that it has a nef divisor that is not semiample. Recently, Chen and Coskun proved in \cite{CC} that  $\oM_{1,n}$ is not a MDS for $n\geq3$ as it has infinitely many extremal effective divisors. For genus zero, the only previously settled cases were for $n\leq6$ ($\oM_{0,5}$ is a del Pezzo surface, hence, a MDS by \cite{BP}; $\oM_{0,6}$ is log-Fano threefold, hence, a MDS by \cite{HK}; for a direct proof that $\oM_{0,6}$ is a MDS, see \cite{C}. Note more generally that in characteristic zero, log-Fano varieties are MDS by \cite{BCHM}; however, $\oM_{0,n}$ is not log-Fano for $n\geq7$).  Since \cite{HK}, the question whether $\oM_{0,n}$ is a MDS was raised by several authors, see for example 
\cite{C}, 
\cite{AGS}, \cite{GM}, 
\cite{Kiem}, 
\cite{McK_survey}, 
\cite{Fed}, 
\cite{Hausen}, \cite{GianGib}, \cite{Milena}, \cite{GM_nef}, 
\cite{BGM},  
\cite{CT1}, \cite{GianJenMoon}, 
\cite{Larsen}.
One of the results in \cite{Milena} is that $\oM_{0,n}$ is a MDS if and only if the projectivization of the pull-back of the cotangent bundle of $\bP^{n-3}$ to $\oLM_n$ is a MDS. In particular, Cor. \ref{main} adds to the examples in \cite{Milena} of toric vector bundles whose projectivization is not a MDS. 

\smallskip

The original motivation for Hu and Keel's question was coming from Keel and M\textsuperscript{c}Kernan's result \cite{KM} that any extremal ray of the Mori cone of $\oM_{0,n}$ that (1) can be contracted by a map of relative Picard number $1$ and (2)
the exceptional locus of the map in (1) has dimension at least $2$,  is generated by a one-dimensional stratum  (i.e., the Fulton-Faber conjecture is satisfied for such rays). As in a MDS any extremal ray of the Mori cone can be contracted by a map of relative Picard number $1$, a positive answer to the Hu-Keel question ``would nearly answer Fulton's question for  $\oM_{0,n}$" \cite{HK}. Implicit in this statement is the expectation that condition (2) should be satisfied.  It was a long held belief that the exceptional locus of any map $\oM_{0,n}\to X$ has all components of dimension at least $2$. We gave counterexamples to this statement in \cite{CT2}.

\begin{remarks}

(1) By the Kapranov description, $\oM_{0,n}$ is the iterated blow-up of $\bP^{n-3}$ along proper transforms of linear subspaces spanned by $n-1$ points in linearly general position. The Losev-Manin space $\oLM_n$ is the iterated blow-up of $\bP^{n-3}$ along proper transforms of linear subspaces spanned by $n-2$ points in linearly general position. We denote by $X_n$ the intermediate toric variety  obtained by blowing-up only linear subspaces of codimension $\geq3$. By Cor. \ref{codim 3}, $\Bl_e X_{n+1}$ is a small modification of a certain $\bP^1$-bundle over $\oM_{0,n}$. In particular, $\Bl_eX_{n+1}$ is not a Mori Dream Space if  $\ch k=0$ and $n\ge 134$.

(2) Thm. \ref{GNW theorem} is stated slightly differently in \cite{GNW}. In Section \ref{GNW section} we translate into a geometric proof the arguments in \cite{GNW}. They are based on  reduction to positive characteristic
and a version of Max Noether's ``AF+BG'' theorem that holds for weighted projective 
planes.

(3) Several arguments in this paper involve elementary transformations of vector bundles, for example
the second part of Thm. \ref{asdazxvsfvsfvsdaqwf} follows by doing elementary tranformations of rank $2$ bundles on $\oM_{0,n}$. We give a general criterion for being able to iterate elementary transformations (Prop.~\ref{mvcvcvcnb}),
which might be of independent interest.
\end{remarks}

{\bf Acknowledgements.} We are grateful to Aaron Bertram, Tommaso de Fernex, Jose Gonzalez, Sean Keel and James M\textsuperscript{c}Kernan for useful discussions. We thank the referee for several useful comments. The first author was partially supported by NSF grants DMS-1160626 and DMS-1302731. The second author was partially supported by NSF grants DMS-1001344 and DMS-1303415.

\section{Preliminaries}

We briefly recall some basic properties of MDS from \cite{HK}. 

Let $X$ be a normal projective variety. A \emph{small $\bQ$-factorial modification} (SQM for short) of $X$ is a small (i.e., isomorphic in codimension one) birational map $X\dashrightarrow Y$ to another normal $\bQ$-factorial projective variety $Y$. 

\begin{definition}\label{MDS def}
A normal projective variety $X$ is called a \emph{Mori Dream Space (MDS)} if the following conditions hold:
\begin{itemize}
\item[(1) ] $X$ is $\bQ$-factorial and $\Pic(X)_{\bQ}\cong \NN^1(X)_{\bQ}$; 
\item[(2) ] $\Nef(X)$ is generated by finitely many semi-ample line bundles;
\item[(3) ] There is a finite collection of SQMs $f_i: X\dashrightarrow X_i$ such that each $X_i$ satisfies (1), (2) and $\Mov(X)$ is the union of $f_i^*(\Nef(X_i))$.
\end{itemize}
\end{definition}


\begin{say}\label{MDS obs}
In what follows, we will often make use of the following facts:
\begin{itemize}
\item If $X$ is a MDS, any normal projective variety $Y$ which is an SQM of $X$, is also a MDS. 
This follows from the fact that the $f_i$ of Def. \ref{MDS def} are the only SQMs of $X$ \cite[Prop. 1.11]{HK}. 

\item  (\cite[Thm. 1.1]{Okawa})  
Let $X\to Y$ be a surjective morphism of projective normal $\bQ$-factorial varieties.
If $X$ is a MDS then $Y$ is a MDS. Note, we only use this for maps $f$ with connected fibers, in which case the statement follows from \cite{HK}. 
\end{itemize}
\end{say}

\begin{definition}\label{Cox}
For a semigroup $\Gamma$ of Weil divisors on $X$, consider the $\Gamma$-graded ring: 
$$R(X,\Gamma):=\bigoplus_{D\in\Gamma}\HH^0(X, \O(D)).$$
where $\cO(D)$ is the divisorial sheaf associated to the Weil divisor $D$. 
Suppose that the divisor class group $\Cl(X)$ is finitely generated. If $\Gamma$ is a group of Weil divisors such that $\Gamma_{\bQ}\cong\Cl(X)_{\bQ}$, the ring $R(X, \Gamma)$ is called a \emph{Cox ring} of $X$ and is denoted $\Cox(X)$. 
\end{definition}

The definition of $\Cox(X)$ depends on the choice of $\Gamma$, but finite generation of $\Cox(X)$ does not. Def. \ref{Cox} differs from  \cite[Def. 2.6]{HK}, in that $\Cl(X)$ replaces $\Pic(X)$. However, for us $X$ will always be $\bQ$-factorial; hence, 
finite generation of  $\Cox(X)$ is not affected. The following is an algebraic characterization of MDS:

\begin{theorem}\cite[Prop. 2.9]{HK}
Let $X$ be a $\bQ$-factorial projective variety with $\Pic(X)_{\bQ}\cong\NN^1(X)_{\bQ}$. 
Then $X$ is a MDS if and only if $\Cox(X)$ is a finitely generated $k$-algebra. 
\end{theorem}
 

\section{Proof of Theorem \ref{asdasdaqwf}}\label{main section}

\begin{proposition}\label{toric}
Let $\pi:\,N\to N'$ be a surjective map of lattices 
(finitely generated free $\Z$-modules)
with kernel of rank $1$ spanned by a primitive vector $v_0\in N$. 
Let $\Gamma$ be a finite set of rays in $N_{\bR}$ spanned by elements of $N$, 
such that the rays $\pm{R_0}$ spanned by $\pm{v_0}$ are not in~$\Gamma$. Let $\cF'\subset N'_\bR$ be a complete simplicial fan with rays given by $\pi(\Gamma)$. 
Suppose that the corresponding toric variety $X'$ is projective
(notice that it is also $\bQ$-factorial because $\cF'$ is simplicial). Then

(A) There exists a complete simplicial fan $\cF\subset N_\bR$
with rays given by $\Gamma\cup\{\pm R_0\}$ and such that
\begin{itemize}
\item the corresponding toric variety $X$ is projective;
\item the rational map $p:\,X\dashrightarrow X'$ induced by $\pi$ is regular;
\item each cone of $\cF$ maps onto a cone of $\cF'$.
\end{itemize}

(B) There exists an SQM $Z$ of $\Bl_eX$ such that the 
rational map $Z\dashrightarrow\Bl_eX'$ induced by $p$ is regular.  In particular, if $\Bl_eX$ is a MDS then $\Bl_eX'$ is a MDS.
\end{proposition}

\begin{proof}
We first prove (A). We argue by induction on $|\Gamma|-|\pi(\Gamma)|$. Suppose that this number is zero, and in particular we have a bijection between $\Gamma$ and $\pi(\Gamma)$.
Then we define $\cF$ as follows: for any subset $J\subset\Gamma$ (maybe empty)
such that the rays spanned by the vectors in $\pi(J)$ form a cone, $\cF$ contains the cone spanned by the rays in $J$, the cone spanned by the rays in $J\cup\{R_0\}$, and the cone spanned by the rays in $J\cup\{-R_0\}$. It follows from the fact that $\cF'$ is a complete simplicial fan that $\cF$ is a also a complete simplicial fan $\cF\subset N_\bR$ with rays in $\Gamma\cup\{\pm R_0\}$. Moreover, the rational map 
$p:\,X\dashrightarrow X'$ induced by $\pi$ is regular and in fact
each cone of $\cF$ maps onto a cone of $\cF'$. 

Next we show that $X$ is projective.  
It follows from the description of the map of fans that
all fibers of $p$ are $\bP^1$'s (only set-theoretically because the fibers are not necessarily reduced), and moreover~$D_{0}$, the torus invariant $\bQ$-Cartier divisor corresponding to the ray~$R_0$, is a section of $p$ and therefore is $p$-ample.
It follows that $p$ is projective and therefore that $X$ is projective because $X'$ is projective.
For a purely toric proof of projectivity, let $A$ be an ample Cartier divisor on $X'$.
 Let $D=D_0+mp^*(A)$. We argue that the $\bQ$-Cartier divisor 
$D$ is ample for large $m>0$ by using the Toric Kleiman Criterion \cite[Thm. 6.3.13]{CLS}, i.e., we prove that $D\cdot C>0$ for every torus invariant curve $C$ in $X$. Torus invariant curves have the form $V(\tau)$, for $\tau$ a cone in $\cF$ of dimension $n-1$ ($n=\dim X$). There are two cases: (1) $\tau$ is spanned by rays $R_1,\ldots, R_{n-1}$ in $\Gamma$; and (2) $\tau$ is spanned by $R_0$ and rays  $R_1,\ldots, R_{n-2}$ in $\Gamma$. In Case (1), $p(C)$ is a point in $X'$; hence, $D\cdot C=D_0\cdot C$. Note that $\tau=\sigma\cap\sigma'$, where $\sigma$ is the cone spanned by $\tau$ and $R_0$ and $\sigma'$ is the cone spanned by $\tau$ and $-R_0$. Then by \cite[Lemma 6.4.2]{CLS} $D_0\cdot C=\frac{\mult(\tau)}{\mult(\sigma)}>0$, where $\mult(\sigma)$ denotes the multiplicity of a simplicial cone $\sigma$. In Case (2), $p(C)$ is the torus invariant curve $V(\overline{\tau})$ in $X'$, where $\overline{\tau}=\langle\pi(R_1),\ldots, \pi(R_{n-2})\rangle$. Let $M\geq0$ be an integer such that $D_0\cdot C'> -M$, for all the torus invariant curves $C'$ in $X$. By the projection formula, $D\cdot C=D_0\cdot C+ mA\cdot p_*(C)>0$ if $m\geq M$.  



Now we do the inductive step. Let $R'\in\pi(\Gamma)$ and let $Z\subset\Gamma$ be 
the set of all rays $R\in\Gamma$ such that $\pi(R)=R'$. Without loss of generality 
we can suppose that $|Z|>1$. Choose $R\in Z$. Let $\tilde\Gamma=\Gamma\setminus\{R\}$.
Since the rays of $\cF'$ are given by $\pi(\tilde\Gamma)=\pi(\Gamma)$, by the inductive assumption, the theorem is true for $\tilde\Gamma$. Let $\cG\subset N_\bR$ be the corresponding fan and $\tilde{X}$ be the corresponding toric variety. Let $\pi_{\bR}: N_{\bR}\to N'_{\bR}$ be the map induced by $\pi$. Then $\pi_{\bR}^{-1}(R')\subset N'_{\bR}$ is a $2$-dimensional half-space, which is the union of the cones in $\cG$ spanned by pairs of rays:
$$\{R_0=U_0,U_1\}, \{U_1,U_2\},\ldots, \{U_{k-1},U_k\}, \{U_k,U_{k+1}=-R_0\},$$
where $\{U_1,\ldots,U_k\}=Z\setminus\{R\}$ (see Figure \ref{fanpic}). 
\begin{figure}[htbp]
\includegraphics[width=4in]{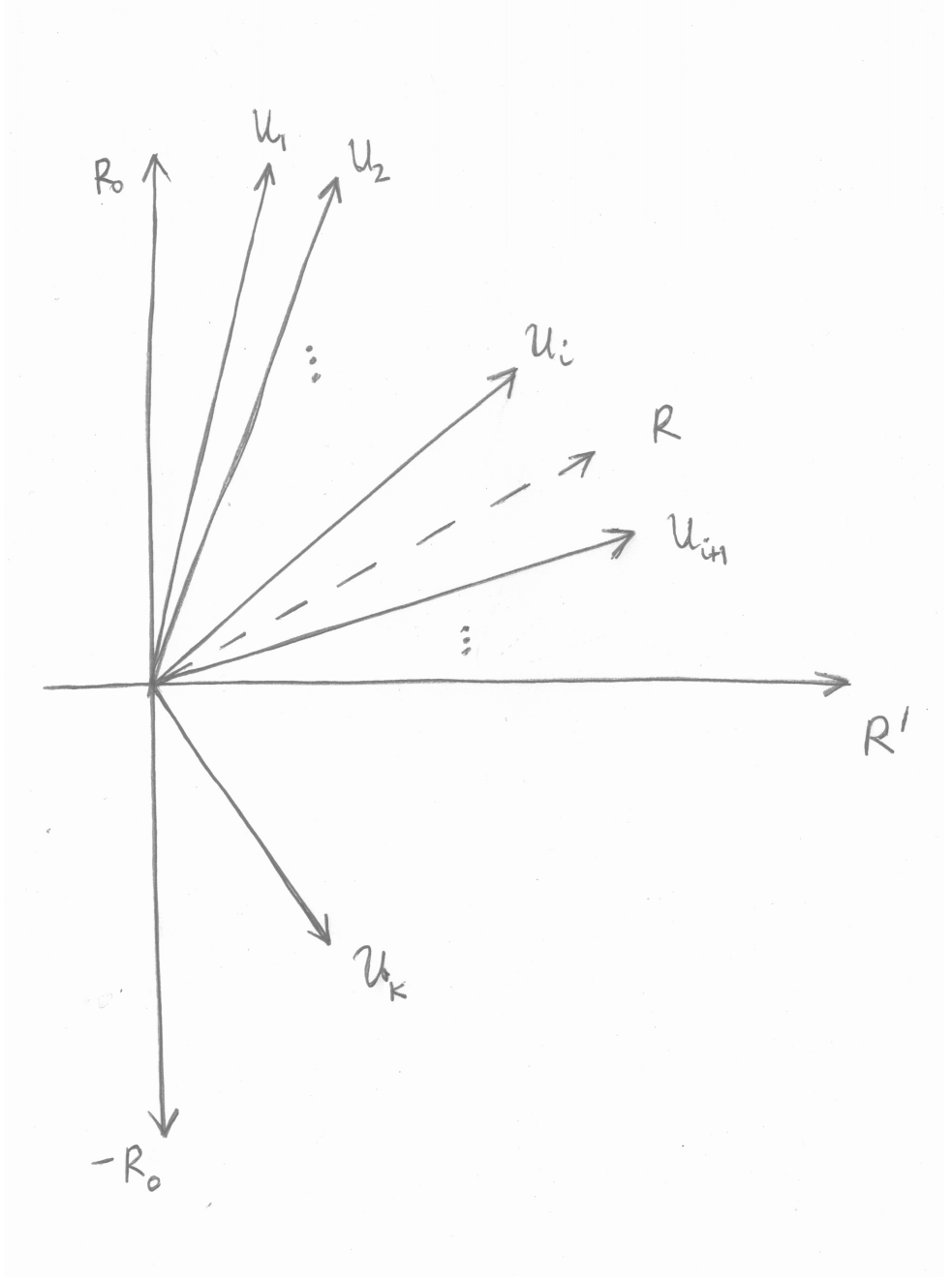}
\caption{\small 
The rays $U_1,\ldots, U_k$ of $\tilde{\Gamma}$ that map to $R'$.}\label{fanpic}
\end{figure}

Choose an index $i$ such that $R$ belongs to 
the relative interior of the angle spanned by $U_i$ and $U_{i+1}$. Then the fan $\cF$ is obtained as a star subdivision on $\cG$ centered at $R$. 
By \cite[Prop. 11.1.6]{CLS} the map $X\to \tilde{X}$ is projective. All properties in (A) are clearly satisfied.

Now we prove (B). 
Notice that the map $p:\,X\to X'$ 
over the open torus $T'\subset X'$ is a trivial $\bP^1$-bundle
$\pr_1:\,T'\times\bP^1\to T'$ (recall that the map $p:\,X\to X'$ is not globally a 
$\bP^1$-bundle). To construct $Z$ and the morphism $f: Z\to\Bl_e X'$ that factors through $\Bl_e X$, we first construct a small modification $Z'$ of $\Bl_e(T'\times\bP^1)$ and a morphism  
$$f': Z'\to \Bl_e T'$$ resolving the induced rational map 
$\Bl_e(T'\times\bP^1)\dashrightarrow \Bl_e T'$.

We then obtain $Z$ and $f$ by gluing $f'$ to 
$$p: X\setminus p^{-1}\{e\}\to X'\setminus\{e\}$$ along
the $\bP^1$-bundle $\pr_1:\,(T'\setminus\{e\})\times\bP^1\to (T'\setminus\{e\})$. 

To construct $Z'$, we do a linear change of variables to identify 
$$T'\simeq\bA^k\setminus\bigcup_i\{x_i=-1\},\quad e\mapsto 0$$
and
$$\bP^1\simeq\bP^1,\quad 1\mapsto 0.$$
Thus we identify $p_{|p^{-1}(T')}$ with the restriction of the toric projection map
$\pr_1:\,\bA^k\times\bP^1\to \bA^k$
(for a different choice of the toric structure) to the open set $T'\subset\bA^k$.
Blow-ups of $X$ and $X'$ at the identity elements of their tori now correspond to 
blow-ups  in torus fixed points:
$$Y:=\Bl_0{\bA^k\times\bP^1},\quad Y':=\Bl_{0}\bA^k$$

The fans are as follows: the fan of $Y'$ is the star subdivision of the positive octant $\langle e_1,\ldots,e_k\rangle$ in the vector $e_0:=e_1+\ldots+e_k$. Its top-dimensional cones are spanned by $e_0$ and $\{e_i\}_{i\in I}$, where $I\subset\{1,\ldots,k\}$ is a subset of cardinality $k-1$.
The fan of $Y$ contains an octant $\tau=\langle e_1,\ldots,e_k,-e_{k+1}\rangle$
and the star subdivision of the positive octant $\langle e_1,\ldots,e_k,e_{k+1}\rangle$
in the vector $f_0:=e_1+\ldots+e_{k+1}$. In particular, the fan of $Y$ contains the cone
$\tau'=\langle e_1,\ldots,e_k,f_0\rangle$. We construct a small modification $Z'$ of $Y$
as follows: We remove the cones $\tau$ and $\tau'$ from the fan of $Y$
and instead add $k$ top-dimensional cones spanned by $f_0$, $-e_{k+1}$, and 
$\{e_i\}_{i\in I}$, where $I\subset\{1,\ldots,k\}$ is a subset of cardinality $k-1$. To see this geometrically, consider the trivial bundle  $\bP:=Y'\times\bP^1\to Y'$ with its sections $s_0=Y'\times\{0\}$ and $s_{\infty}=Y'\times\{\infty\}$. If $E$ denotes the exceptional divisor of $Y'\to\bA^k$, let $Z=s_0(E)$. Let $\tilde\bP$ be the blow-up of $\bP$ along $Z$. Let $D=E\times\bP^1\subset\bP$ and  let $\tilde D$ be its proper transform in $\tilde\bP$. There are two ways to blow-down $\tilde D\cong\bP^{k-1}\times\bP^1$: 
$$\al: \tilde\bP\to Z',\quad \al(\tilde D)=\tilde{s}_{\infty}(E)\cong\bP^{k-1},$$
$$\be: \tilde\bP\to Y,\quad \be(\tilde{D})=\tilde{F}\cong\bP^1,\quad F=\{0\}\times\bP^1$$  
where $\tilde{s}_{\infty}$ is the proper transform of the section $s_{\infty}$ under the rational map $\bP\dashrightarrow Z'$ and $\tilde{F}$ is the proper transform of $F$ in $Y$. Notice that the rational map $Z'\dashrightarrow Y'$ is regular, and one can check that it is the $\bP^1$-bundle $\bP_{Y'}(\cO\oplus\cO(-E))$. Note that over $Y'\setminus E\cong(\bA^k\setminus\{0\})\times\bP^1$ all above birational maps are isomorphisms.  

\begin{remark}
Note that $Z'$ is the elementary transformation of the trivial 
$\bP^1$-bundle over $Y'$ given by the data $(E,Z)$ (see Section \ref{iff section}). 
Alternatively, one can construct $Z'$ and $f'$ by doing this elementary transformation. Then it is not hard to argue that the new $\bP^1$-bundle is a small modification of 
$\Bl_e(T'\times\bP^1)$.  
\end{remark}

To construct $Z$ and the morphism $f: Z\to\Bl_e X'$, we glue $Z'\to Y'$  (with preimages of hyperplanes $\{x_i=-1\}$ removed) to 
$$p: X\setminus p^{-1}\{e'\}\to X'\setminus\{e'\}$$ along
the $\bP^1$-bundle $\pr_1:\,(T'\setminus\{e\})\times\bP^1\to (T'\setminus\{e\})$. 
Clearly, $Z$ is $\bQ$-factorial, since $Z'$ and $X$ are $\bQ$-factorial.   

It remains to show that $Z$ is projective and it would suffice to show that the morphism $f$ is projective. This morphism is clearly projective in both charts of $Z$, but since projectivity is not local on the base, we have to give a global argument. It is enough to construct an $f$-ample divisor on $Z$. Let $A$ be an irreducible very ample divisor on $X$ and let $\tilde A$ be its proper transform in $Z$. We claim that $\tilde A$ is $f$-ample. Indeed, it is obviously $f$-ample in the second chart of $Z$. But the first chart is a $\bP^1$-bundle and $\tilde A$ surjects onto the base, and so it is $f$-ample.
\end{proof}

\begin{proof}[Proof of Thm. \ref{asdasdaqwf}.]
The toric data of $\oLM_n$ is as follows, see \cite{LM}. Fix general vectors $e_1,\ldots,e_{n-2}\in\bR^{n-3}$
such that $e_1+\ldots+e_{n-2}=0$. The lattice $N$ is generated by $e_1,\ldots,e_{n-2}$.
The rays of the fan of $\oLM_n$ are spanned by the primitive lattice vectors $\sum_{i\in I}e_i$, for each subset 
$I$ of $S:=\{1,\ldots, n-2\}$ with $1\le |I|\le n-3$. 
Notice that rays of this fan come in opposite pairs.
We are not going to need cones of higher dimension of this fan.
The main idea is to choose a sequence of projections from these rays
to get a sequence of (generically) $\bP^1$-bundles
$$X_1\to X_2\to X_3\to X_4\to\ldots,$$
where $X_1$ is an SQM of $\oLM_n$ which is different from the 
standard tower of forgetful maps
$$\oLM_n\to\oLM_{n-1}\to\oLM_{n-2}\to\ldots$$
Specifically, we partition 
$$S=S_1\coprod S_2\coprod S_3$$ 
into subsets of size $a+2, b+2, c+2$ (so $n=a+b+c+8$).
We also fix some indices $n_i\in S_i$, for $i=1,2,3$. 
Let $N''\subset N$ be a sublattice spanned by the following vectors:
\begin{equation}\label{axfafgfgfg}
e_{n_i}+e_r\quad\hbox{\rm for}\quad r\in S_i\setminus\{n_i\},\ i=1,2,3.
\cooltag\end{equation}
Let $N'=N/N''$ be the quotient group and let $\pi$ be the projection map.
Then we have the following:
\begin{enumerate}
\item $N'$ is a lattice;
\item $N'$ is spanned by the vectors $\pi(e_{n_i})$, for $i=1,2,3$;
\item $a\pi(e_{n_1})+b\pi(e_{n_2})+c\pi(e_{n_3})=0$.
\end{enumerate}
It follows at once that the toric surface with lattice $N'$ and rays spanned by $\pi(e_{n_i})$ for $i=1,2,3$,  is a weighted projective plane $\bP(a,b,c)$. 

To finish the proof of the theorem, we apply Prop. \ref{toric} inductively to the sequence of lattices $N_j$, $j=1,\ldots,n-4$, obtained by taking the quotient of $N$ by the sublattice spanned by the first $j-1$ vectors of the sequence \eqref{axfafgfgfg}
(arranged in any order) and the sets of rays $\Gamma_j$ obtained by projecting the rays of the fan of $\oLM_n$.
More precisely, we do a backwards induction, by starting with the canonical simplicial structure 
on the fan of the complete (hence, projective) toric surface $X_{n-4}$ with data $N'=N_{n-4}$, $\Gamma_{n-4}$. It remains to notice that we have a regular map $X_{n-4}\to\bP(a,b,c)$ obtained by dropping all vectors in $\Gamma_{n-4}$ except for $\pi(e_{n_i})$ for $i=1,2,3$. Clearly, the map is an isomorphism on the open torus; hence, there is a birational morphism $\Bl_e X_{n-4}\to\Bl_e\bP(a,b,c)$. The result of applying induction is a sequence of toric morphisms 
$$X_1\to X_2\to\ldots\to X_{n-4},$$
such that the rational map $\Bl_e X_i\dashrightarrow\Bl_e X_{i+1}$ factors through a projective $\bQ$-factorial small modification $Z_i$ of  $\Bl_e X_i$, followed by a surjective regular map $Z_i\to\Bl_e X_{i+1}$. The first toric variety in the sequence $X_1$ is a small modification of $\oLM_n$ (having the same rays) which is an isomorphism on the open torus. Hence, 
$\Bl_e X_1$ is a small modification of $\Bl_e\oLM_n$. The result now follows from  Thm. \ref{GNW theorem} and \ref{MDS obs}. 
\end{proof}



\section{Proof of Theorem \ref{GNW theorem}}\label{GNW section}

The results in  \cite{GNW} are stated in a slightly different form than Thm. \ref{GNW theorem}.
We first explain how our formulation is equivalent to \cite[Cor. 1.2]{GNW}. For the reader's convenience, we also translate the arguments in  \cite{GNW} into a geometric proof of Thm. \ref{GNW theorem}.

\medskip

Let $a, b, c>0$ be pairwise coprime integers. Let $\bP:=\bP(a,b,c)$ be the weighted projective space $\Proj k[x,y,z]$, with $\deg(x)=a$, $\deg(y)=b$, $\deg(z)=c$. Then $\bP$ is a toric variety which is smooth outside the three torus invariant points. Consider the  torus invariant divisors:
$$D_1=V_+(x), \quad D_2=V_+(y), \quad D_3=V_+(z).$$

Let $m_i$ ($i=1,2,3$) be integers such that $m_1a+m_2b+m_3c=1$ and let $H=\sum m_i D_i$. Then $\Cl(\bP)=\bZ\{H\}$, $H$ is $\bQ$-Cartier and $H^2=1/(abc)$. 

Let $\gp:=\gp(a,b,c)$ be the kernel of  the $k$-algebra homomorphism:
$$\phi: k[x, y, z]\rightarrow k[t], \quad\phi(x)=t^a,\quad\phi(y)=t^b,\quad\phi(z)=t^c.$$

The identity of the open torus in $\bP$ is the point $e=V_+(\gp)$. Let $X=\Bl_e\bP$ denote the blow-up of $\bP$ at $e$; let $E$ denote the exceptional divisor. As $e\notin D_i$, we can
pull-back to $X$ the Weil divisors $D_i$ and let $A=\sum m_i\pi^{-1}(D_i)$. Then $\Cl(X)=\bZ\{A, E\}$. A Cox ring of $X$ is:
$$\Cox(X)=\oplus_{d,l\in\bZ}\HH^0(X, \O(dA-lE)). $$
Note that since $a,b,c$ are pairwise coprime, $\cO(dH)\cong\cO(d)$. 

\smallskip

It was observed by Cutkosky \cite{Cutkosky} that finite generation of $\Cox(X)$ is equivalent to the finite generation of the 
symbolic Rees algebra $R_s(\gp)$ (here we follow the exposition in \cite{Kurano-Matsuoka}). Recall that for a prime ideal $\gp$ in a ring $R$, the $l$-th \emph{symbolic power} of $\gp$ is the ideal:
$$\gp^{(l)}=\gp^lR_{\gp}\cap R.$$
The subring of the polynomial ring $R[T]$ given by
$$R_s(\gp):=\bigoplus_{l\geq0}\gp^{(l)}T^l,$$
is called the  \emph{symbolic Rees algebra} of $\gp$. 

In our situation, for the prime ideal $\gp$ in $S=k[x,y,z]$ defined above, we identify  the symbolic Rees algebra $R_s(\gp)$ with a subalgebra of $\Cox(X)$. Using the identification $\HH^0(\bP,\O(d))=S_d$, we have:
$$\HH^0(X, \O(dA-lE))\cong\HH^0(\bP, \O(d)\otimes\I^l_e)=S_d\cap\gp^{(l)},$$ 
where $\I_e$ denotes the ideal sheaf of the point $e$. It follows that $R_s(\gp)$ is isomorphic to the subalgebra of $\Cox(X)$ given by $$\bigoplus_{d,l\geq0}\HH^0(X,\O(dA-lE)).$$
Moreover, $\Cox(X)$ is isomorphic to the extended symbolic Rees ring:
$$R_s(\gp)[T^{-1}]=\ldots\oplus ST^{-2}\oplus ST^{-1}\oplus S\oplus \gp T\oplus\gp^{(2)}T^2\oplus\ldots$$
Clearly, $R_s(\gp)$ is a finitely generated $k$-algebra if and only if $\Cox(X)$~is.

Assume now that
$$(a,b,c)=(7m-3, 5m^2-2m, 8m-3), \quad m\geq4,\quad m\not\equiv0 \mod 3.$$ 

By \cite[Cor. 1.2]{GNW}, the symbolic Rees algebra $R_s(\hat{\gp})$ of the extended ideal $\hat{\gp}$ in the formal power series ring $\hat{S}=k[[x,y,z]]$ is not Noetherian if $\ch k=0$
 (and it is Noetherian if $\ch k>0$).  Since 
 $R_s(\gp)\otimes_S\hat{S}\cong R_s(\hat{\gp})$ \cite[Lemma 2.3]{GN_ams}, it follows that $R_s(\gp)$ is not 
 finitely generated.
 Indeed, otherwise $R_s(\hat{\gp})$ would be a finitely generated $\hat S$-algebra and hence
 Noetherian by Hilbert's basis theorem.
 

%

\medskip

We now give a geometric proof of Thm. \ref{GNW theorem}.  First note the following characterization of $X$ being a MDS in the presence of a negative curve:

\begin{lemma}\cite{Huneke,Cutkosky}\label{MDS surface} 
Assume $X=\Bl_e\bP$ contains an irreducible curve $C\neq E$ with $C^2<0$. Then $X$ is a MDS if and only if there exists an effective divisor $D$ such that $D\cdot C=0$ and 
$D$ does not contain $C$ as a fixed component. 
\end{lemma}

\begin{proof}
Since $C^2<0$, it follows that $C$ generates an extremal ray of the Mori cone $\NE(X)$ and hence, $\NE(X)=\bR_{\geq0}\{C, E\}$. The nef cone is generated by $H$ and the ray $R$ in $\NE(X)$ defined by $R\cdot C=0$, $R\cdot E>0$. Then $X$ is a MDS if and only if $R$ is generated by a semiample divisor. This proves the ``only if" implication. If there is an effective divisor $D$ as in the lemma, we may replace $D$ with a divisor that has no fixed components and $D$ is semiample by Zariski's theorem (\cite[2.1.32]{Laz}).
\end{proof}

\begin{remark}
As observed by Cutkosky \cite{Cutkosky}, if $\ch k>0$ and $X=\Bl_e\bP$ contains a negative curve, then $X$ is always a MDS due to Artin's contractability criterion \cite{Artin}. 
\end{remark}

Let now $(a,b,c)=(7m-3, 5m^2-2m, 8m-3)$, $m\geq4$, $m\not\equiv0 \mod 3$. 
Let $C$ be the proper transform on $X$ of the curve $y^3=x^mz^m$ in $\bP$. The class of $C$ in $\Cl(X)$ is 
$$C=3(5m^2-2m)H-E.$$ 
Note that $C$ is an irreducible curve with $C^2<0$. If $D\in\NE(X)$ is such that $D\cdot C=0$, the class of $D$ equals
$$D_d:=d(7m-3)(8m-3)H-3dE,$$
for some positive integer $d$. 

\medskip

Consider the set $\I$ of effective Weil divisors $D$ on $X$ such that  $D\cdot C=0$ and $D$ does not contain $C$ as a fixed component. A crucial fact is the following:

\begin{proposition}\cite{GNW}\label{MaxNoether} 
The set 
$$I=\{d\in\bZ_{\ge0}\quad |\quad \exists D\in\I,\ [D]=D_d\}$$
equals $\bZ_{\ge0}d_0$ for some non-negative integer $d_0$. 
\end{proposition}

We will prove Prop.~\ref{MaxNoether} using a version of Max Noether's ``AF+BG" theorem \cite[p. 61]{Fulton} that holds for weighted projective planes. Note that  $\I$ and $I$ depend on the field $k$. 
We will write $\I_k$ whenever we need to specify the field $k$. 

\begin{definition}
Let $f, g\in S$ and  $\gp$ be a prime ideal in $S$ which is a minimal prime of the ideal $(f,g)$. We say that \emph{$h\in S$ satisfies Noether's condition at the prime ideal $\gp$ (with respect to $f,g$)} if $h\in(f,g)S_{\gp_i}$. 
\end{definition}

\begin{proposition}[AF+BG theorem]\label{MaxNoether WPP}
Let $f, g, h\in S$. Assume that the minimal primes $\gp_1,\ldots, \gp_s$ of the ideal $(f,g)$ all have height $2$. If $h$ satisfies Noether's condition at $\gp_i$ for all $i=1,\ldots, s$, then $h\in(f,g)$. 
\end{proposition}

\begin{proof}
As $h\in(f,g)S_{\gp_i}$, there exist $u_i\in S\setminus\gp_i$ such that $u_ih\in(f,g)$.  
For each $i$ we can find elements 
$y_i\in \cap_{j\neq i}\gp_j\setminus\gp_i$. Then $u:=\sum u_i y_i\notin\gp_i$ for any $i$ and $uh\in (f,g)$. Since $S$ is Cohen-Macaulay, by the Unmixedness Theorem \cite[Cor. 18.14]{Eisenbud}, all the associated primes of $(f,g)$ are minimal. Hence, the zero divisors of $S/(f,g)$ consist of elements from $\gp_i$'s. It follows that $u$ is not a zero divisor in $S/(f,g)$, hence
$h\in (f,g)$. 
\end{proof}

\begin{corollary}\label{MN}
If $F=V_+(f)$, $G=V_+(g)$ are curves in $\bP$ with no common components and $h\in S$ satisfies Noether's condition at each point of $F\cap G$, then $h=Af+Bg$, for some $A,B\in S$. 
\end{corollary}

\begin{lemma}\label{N condition}
Assume $F=V_+(f)$ and $G=V_+(g)$ are curves in $\bP$ with no common components, $F\cap G$ does not contain any of the torus invariant points, and $F$ is smooth along $F\cap G$. Let  $h\in S$ and let $G'=V_+(h)$. Assume that for all $p\in F\cap G$ we have:
$$\mult_p(G',F)\geq\mult_p(G,F).$$ 
Then $h$ satisfies Noether's condition at each point of $F\cap G$. 
\end{lemma}

\begin{remark}
Note that this lemma includes the ``classical'' 
case when $F$ and $G$ intersect transversally (and away from torus fixed points)
and $G'$ passes through all points in $F\cap G$.
\end{remark}

\begin{proof}
Let $p\in F\cap G$ with the corresponding homogeneous prime ideal~$\gp$. By assumption, at least two of $x,y,z$ are not in $\gp$. Say $x,y\notin\gp$. Since $a,b$ are coprime, let $m_1,m_2$ be integers such that $m_1a+m_2b=1$. Let $r=x^{m_1}y^{m_2}$. Note that $r$ is a unit in $S_{xy}$.  For $f\in S_d$, denote $f_1=f/r^d\in S_{(xy)}$. Consider the functions $f_1, g_1,  h_1$ corresponding to $f,g,h$. Denote by $t$ a generator of the maximal ideal of  $\cO_{C,p}=\cO_{\bP,p}/(f_1)$. If $\overline{g}_1$, $\overline{h}_1$ denote the images of $g_1$, $h_1$ in $\cO_{C,p}$, we have $\overline{g}_1=ut^n$, $\overline{h}_1=vt^m$, for units $u,v\in\cO_{C,p}$ and with $n=\mult_p(G,F)$,  $m=\mult_p(G',F)$. As $m\geq n$, it follows that $\overline{h}_1\in(\overline{g}_1)$, i.e., $h_1\in(f_1,g_1)\subseteq\cO_{\bP,p}=S_{(\gp)}$. Since $x,y\notin\gp$, it follows that $h\in(f,g)S_{\gp}$. 
\end{proof}

\begin{proof}[Proof of Prop. \ref{MaxNoether}]
Assume $\I\neq\emptyset$ and let $d_0$ be the smallest positive integer in $I$. Let $g\in S$ be such that the proper transform $D$ in $X$ of $G:=V_+(g)\subset\bP$ has class $D_{d_0}$ and such that $D$
does not contain $C$.
Let $d\in I$, $d>0$. Let $h\in S$ be such that the proper transform $D'$ of $G':=V_+(h)$ has class $D_{d}$
 and such that $D'$ does not contain $C$.
Recall that $C$ is the proper transform in $X$ of $F:=V_+(f)$, where $f=y^3-x^mz^m$. Since $D\cdot C=0$, $D$ and $C$ are disjoint in $X$, but $G$ and $F$ intersect only at $e$ in $\bP$ and we have:
$$\mult_e(G,F)=\mult_e(G)=3d_0.$$ 

Similarly, $\mult_e(G',F)=3d$. Since $d\geq d_0$, by Lemma \ref{N condition},  $h$ satisfies Noether's condition (with respect to $f,g$). By Cor. \ref{MN}, $h=Af+Bg$ for some  $A,B\in S$.  If $D_1$ denotes the proper transform in $X$ of $V_+(B)$, note that $[D']=[D]+[D_1]$. It follows that $D_1\in\I$ and so $d-d_0\in I$. The statement now follows by induction.  
\end{proof}

\begin{lemma}\cite{GNW}\label{D_p}
 Assume $\ch k=p\geq3$. Then there exists $D\in\I_k$ with class $D_p$. 
\end{lemma}

\begin{proof}
We recall from \cite[p. 390]{GNW}  the construction of a polynomial $h\in\gp^{(3p)}$ of degree $p(7m-3)(8m-3)$ such that  $c\nmid h$. The ideal $\gp$ contains  polynomials $u$, $v$ and $f$, where 
$$u=z^{3m-1}-x^{2m-1}y^2,\quad  v=x^{3m-1} -yz^{2m-1}, \quad f=y^3-x^mz^m$$
(in fact $u,v,f$ generate $\gp$ by the Hilbert--Burch theorem but we don't need this).
Let 
$$d_2=x^{m-1}y^5z^{m-1}-3x^{2m-1}y^2z^{2m-1}+x^{5m-2}y+z^{5m-2},$$
$$d_3=-x^{3m-2}y^7+2x^{m-1}y^5z^{3m-1}+x^{4m-2}y^4z^m-5x^{2m-1}y^2z^{4m-1}+$$
$$+3x^{5m-2}yz^{2m}-x^{8m-3}z+z^{7m-2},$$
$$d'_3=y^8z^{2m-2}-4x^my^5z^{3m-2}+x^{4m-1}y^4z^{m-1}+6x^{2m}y^2z^{4m-2}-$$
$$-4x^{5m-1}yz^{2m-1}+x^{8m-2}-xz^{7m-3}.$$

A direct computation shows:
$$x^md_2-yv^2+z^{m-1}uf=0,$$
$$x^{m-1}v^2f+ud_2-z^{m-1}d_3=0,$$
$$xd_3+yvf^2+zd'_3=0.$$

It follows that $d_2\in\gp^{(2)}$ and $d_3, d'_3\in\gp^{(3)}$. Note also that $f\nmid d_3$. Since $\ch k =p$, we get from the third equation that
$$x^pd_3^p+y^pv^pf^{2p}+z^p{d'_3}^p=0.$$

Write $p=2q+1$ for some integer $q>0$. Since 
$$x^pd_3^p+y^pv^pf^{2p}\equiv0 \mod (z^p),\quad x^mu+y^2v+z^{2m-1}f=0,$$
 
it follows that 
$$x^pd_3^p+y^pv^pf^{2p}=x^pd_3^p+(-1)^qyv^{p-q}f^{2p}\big(x^mu+z^{2m-1}f\big)^q=$$
$$=x^pd_3^p+(-1)^q\sum_{i=0}^q{q \choose{i}}x^{m(q-i)}yz^{(2m-1)i}u^{q-i}v^{p-q}f^{2p+i}$$
$$\equiv0 \mod (z^p).$$ 

Notice that either $m(q-i)\geq p$ or $(2m-1)i\geq p$ for each $0\leq i\leq q$  (use $m\geq4$). Then 
$$x^pd_3^p+(-1)^q\sum_{(2m-1)i<p}{q \choose{i}}x^{m(q-i)}yz^{(2m-1)i}u^{q-i}v^{p-q}f^{2p+i}
\equiv0 \mod (z^p),$$ and therefore,
$$z^ph=d_3^p+(-1)^q\sum_{(2m-1)i<p}{q \choose{i}}x^{m(q-i)-p}yz^{(2m-1)i}u^{q-i}v^{p-q}f^{2p+i},$$
for some $h\in\gp^{(3p)}$. If $f\mid h$, then $f\mid d_3$, which is a contradiction. 
\end{proof}

\begin{proof}[Proof of Thm. \ref{GNW theorem}]
Assume that $X$ is a MDS in characteristic $0$. By Lemma \ref{MDS surface}, there exists an integer $d>0$ and a monic polynomial  $f\in S$ such that the proper transform $D$ in $X$ of $V_+(f)$ has class $D_d$ and $D$ does not contain $C$ as a fixed component. Since a multiple of $D$ is base-point free and $D$ is big, by eventually replacing $d$ with a multiple, we may assume by Bertini's theorem that $D$ is smooth and connected. 

Let $R$ be the $\bZ$-algebra generated by the coefficients of $f$. Let $\bP_R:=\Proj R[x,y,z]$ and $e_R$ be the section of  $\bP_R\rightarrow\Spec(R)$ corresponding to $\gp R[x,y,z]$. Let $\X_R$ be the blow-up of $\bP_R$ along $e_R$, with exceptional divisor $\cE$. Let $\sD$ be the proper transform of $V_+(f)\subset\bP_R$ in $\cX_R$. Since the geometric generic fiber of 
$\rho:\sD\rightarrow\Spec(R)$ is smooth and connected, by eventually replacing $R$ with a localization, we may assume that $\rho$ is smooth and all its geometric fibers $\sD_s$ are connected. Since $\rho$ is flat, $\deg\O(\cE)_{|\sD_s}$  does not depend on $s$. It follows that 
all $\sD_s$ have class $D_d$ and do not contain the curve $C$, i.e., for each $s\in\Spec(R)$, we obtain a divisor in $\cI_{\overline{k(s)}}$. For each prime $p$ in the image of the dominant map $\Spec R\rightarrow\Spec\bZ$, pick some $s_p\in\Spec(R)$. By Prop. \ref{MaxNoether},  there are integers $d_p$ such that $I_{\overline{k(s_p)}}=\bN\{d_p\}$.  Hence,  $d_p\mid d$ for sufficiently large primes $p$. As by Lemma \ref{D_p}, $d_p\mid p$ for all primes $p\geq3$, we must have that $d_p=1$ for all sufficiently large $p$. 

But one can see directly that $D_1$ is not effective in characteristic $0$ (and hence, in characteristic $p$, for $p$ large). To see this, note that we have the following:

\begin{claim}\label{monomials}
The only monomials in $S$ of degree $(7m-3)(8m-3)$ are
$$x^{m-1}y^5 z^{3m-2}, x^{4m-2}y^4 z^{m-1}, x^{2m-1}y^2 z^{4m-2},
x^{5m-2}y z^{2m-1}, x^{8m-3},  z^{7m-3}.$$
\end{claim}

\begin{proof}
To simplify notation, we let 
$$a=7m-3,\quad b=5m^2-2m,\quad c=8m-3.$$ 

Consider monomials $x^{\al}y^{\be}z^{\ga}$ of degree $ac$, i.e., with
$a\al+b\be+c\ga=ac.$ Since $3b=(a+c)m$, it follows that
$$a(3\al+m\be)+c(3\ga+m\be)=3ac.$$

In particular, $a | 3\ga+m\be$ and $c | 3\al+m\be$. 

Moreover, note that $0\leq\al\leq c$, $0\leq\ga\leq a$ and 
$$0\leq\be\leq  \frac{ac}{b}=\frac{(7m-3)(8m-3)}{5m^2-2m}<12.$$

If $\be=0$ then $a | 3\ga$, $c | 3\al$. Since $a$, $c$ are not divisible by $3$, it follows that 
$a | \ga$, $c | \al$ and therefore the only solutions are $\al=c, \ga=0$ and $\al=0, \ga=a$. 

Assume that $\be>0$. Note that for a fixed $\be>0$, there is at most one choice of $\al, \ga$. Indeed, if $a\al_1+c\ga_1=a\al_2+c\ga_2$, it follows from $a(\al_1-\al_2)=c(\ga_2-\ga_1)$ and $(a, c)=1$ that the only possibility is $\al_1=\al_2$ and  $\ga_1=\ga_2$. Moreover, 
as $c | 3\al+m\be$, there is $u\in\bZ_{>0}$ with 
$$cu=3\al+m\be.$$

Since $\al<c$, it follows that $cu<3c+m\be\leq 3c+11m$ and hence, 
$$u<3+\frac{11m}{c}=3+\frac{11m}{8m-3}<3+2=5.$$

Considering divisibility by $3$ in $cu=3\al+m\be$, we must have
$2u\equiv\be$ modulo $3$. Hence, the only possibilities are: $u=1, 4$, $\be=2, 5, 8, 11$; 
 $u=2$, $\be=1, 4, 7, 10$; $u=3$, $\be=3, 6, 9$. For fixed $u$ and $\be$, one computes $\al$ from $cu=3\al+m\be$. One can directly see that the only possibilities are $u=1$, $\be=2, 5$ and  $u=2$, $\be=1, 4$. 
\end{proof}

No linear combination of the monomials in Claim \ref{monomials} has all six derivatives of order $2$ vanishing at $e=(1,1,1)$. A direct computation shows that the determinant of the corresponding $6\times 6$ matrix is (up to a sign): 
 $$4(7m-3)^2(8m-3)^2(7m-4)(8m-4)(51m^2-43m+9).$$
Q.E.D.\end{proof}


\section{Proof of Theorem \ref{asdazxvsfvsfvsdaqwf}}\label{iff section}

We recall the elementary transformations of Maruyama \cite{Mar} in the generality that we need.
Let $X$ be a scheme of finite type over $k$, let $i:\,D\hookrightarrow X$ be an effective Cartier divisor, let $\cF$ be a locally free sheaf of rank $2$ on~$X$, and let $\cF|_D\to \cL$
be a surjection onto an invertible sheaf on~$D$. 
Then we have a commutative diagram: 
$$
\begin{CD}
&&0 && 0 \\
@. @AAA @AAA                 \\
0 @>>> i_*\cL' @>>> i_*(\cF|_D) @>>> i_*\cL @>>>0\\
@. @A{\pi'}AA @AAA                 @|\\
0 @>>> \cF' @>>> \cF @>\pi >> i_*\cL @>>>0\\
@. @AAA @AAA                 \\
&& \cF(-D) @= \cF(-D)                 \\
@. @AAA @AAA                 \\
&&0 && 0 \\
\end{CD}$$

The sheaf $\cF'$ is called an elementary transformation of~$\cF$.
It is a locally free sheaf of rank $2$. 
Geometrically, consider $\bP^1$-bundles $\bP(\cF)$ and $\bP(\cF')$, where say $\bP(\cF)=\bProj_{\cO_X}Sym(\cF)$.
Quotient maps $\pi$ and $\pi'$ give sections $s:\,D\to\bP(\cF|_D)$ and $s':\,D\to\bP(\cF'|_D)$.
Let $Z=s(D)$ and $Z'=s'(D)$ be their images.
Note that they are local complete intersections of codimension $2$.
We have a canonical isomorphism 
$$\Bl_Z\bP(\cF)\simeq\Bl_{Z'}\bP(\cF').$$
More concretely, $\bP(\cF')$ is obtained from $\Bl_Z\bP(\cF)$ by blowing down
the proper transform of the Cartier divisor $\bP(\cF|_D)$. Note that elementary transformations are functorial, i.e., for a map $g: Y \to X$, $\bP(g^*\cF')$ is the elementary transformation of $\bP(g^*\cF)$ along the data $(g^{-1}(D), g^*s)$.

\begin{lemma}\label{aux}
Let $p: Y\to X$ be a $\bP^1$-bundle and let $p': Y'\to X$ be an elementary transformation given by the data $(D, Z)$. Let $t: X\to Y$ be a global section and let $T'$ denote the proper transform of $T=t(X)$ in $Y'$. If $T$ and $Z$ agree over $D$, or if they are disjoint, then $T'$ is a section of $p'$.

Let now $t_1, t_2$ be two global sections and let  $T'_1$, $T'_2$ denote the proper transforms of $T_1=t_1(X)$, $T_2=t_2(X)$.  Assume $T_1$, $Z$ agree over $D$. 

\begin{itemize}
\item[(a) ] If $T_2$, $Z$ are disjoint, then  $T'_1$, $T'_2$ are disjoint over $D$. 
\item[(b) ] Assume $T_1$, $T_2$, $Z$ agree over $D$ and for some point $x\in D$ (with $X$, $D$ non-singular at $x$),  we have at $z=s(x)$ that 
$$T_{z, T_1}\cap T_{z, T_2}=T_{z, Z}\subseteq T_{z, Y}.$$  
Then $T'_1$, $T'_2$ are disjoint over $x$. 
\end{itemize}
\end{lemma}

The condition on tangent spaces in (b) is equivalent to the differentials ${dt_1}_{|x}$, ${dt_2}_{|x}$ not having the same image. Alternatively, there exists a curve $C$ in $X$ smooth at $x$, such that in the ruled surface $S:=p^{-1}(C)\to C$, the sections 
$T_1\cap S$ and $T_2\cap S$ are not tangent at $z$. 

\begin{proof}[Proof of Lemma \ref{aux}]
If $T$ and $Z$ agree along $D$, the proper transform $\tilde{T}$ in the blow-up $\tilde{Y}$ of $Y$ along $Z$ is isomorphic to $T$, as it is the blow-up of $T$ along $Z$ (a Cartier divisor in $T$). As $Y'$ is the blow-down of  $\tilde{Y}$ along the proper transform of 
$p^{-1}(D)$, which is disjoint from $\tilde{T}$, it follows that $T'$ is isomorphic to $\tilde{T}$, hence $T'$ is a section of $p'$.  

Assume that $T$ and $Z$ are disjoint. Set  $Y=\bP(\cF)$,  $Y'=\bP(\cF')$, for $\cF'$ the elementary transformation of $\cF$ along $\cF_{|D}\to\cL$ (corresponding to $Z$). The global section $T$ corresponds to a quotient $\cF\to\cM$. Since $T$ and $Z$ are disjoint,  the induced map $\cF_{|D}\to\cM_{|D}\oplus\cL$ is an isomorphism (hence, the first exact sequence in the commutative diagram relating $\cF$ and $\cF'$ is split). The induced map 
$\cF'\to i_*\cM_{|D}$ factors through $\cF'\to\cF\to\cM$. It follows that $\cF'\to\cM$ is surjective (it is enough to check this on $D$) and $T'=\bP(\cM)$, i.e., $T'$ is a section of $p'$. 

We now prove the second part of the lemma. As proved above, $T'_1$ and $T'_2$ are sections of $p'$. Assume we are in situation (a). We prove that $T'_1$, $T'_2$ are disjoint above any point $x\in D$. Consider a general curve $C$ in $X$ through $x$. By functoriality, the ruled surface $S=p^{-1}(C)\to C$ undergoes an elementary transformation given by data $(x, z)$, where $z=s(x)$. As the section $T_1$ passes through $z$, while $T_2$ does not, it follows immediately that $T'_1$, $T'_2$ are disjoint over $x$. Assume now that we are in situation (b).  As before, we reduce to the ruled surface case. We may choose $C$ a curve through $x$ that is transverse to $D$ at $x$ and let $S=p^{-1}(C)$. It follows that $\dim(T_{z,Z}\cap T_{z,S})=0$ and sections $T_1\cap S$, $T_2\cap S$ are transverse at $z$; hence,  $T'_1$, $T'_2$ are disjoint above $x$. 
\end{proof}

\begin{definition}\label{compatible}
Let $X$ be a non-singular variety and let $D_1,\ldots,D_N$ be irreducible divisors in $X$ with simple normal crossings. Assume that the intersections $D_{ij}:=D_i\cap D_j$ and $D_{ijk}:=D_i \cap D_j\cap D_k$ are irreducible or empty. We denote the interiors of these intersections by $D^0_{ij}$ and $D_{ijk}^0$, respectively. Let $p:\,Y\to X$ be a $\bP^1$-bundle.

A {\em compatible sequence of sections starting at $M$ (with respect to the ordered set $D_1,\ldots,D_N$)} is a sequence $Z_M\ldots, Z_N$, where $Z_i$ is the image of a section $s_i:\,D_i\to p^{-1}(D_i)$ ($i=M,\ldots, N$) 
such that the following conditions are satisfied:
\begin{enumerate}
\item For any $j>i\geq M$, if $D_{ij}\ne\emptyset$ then either 
\begin{enumerate}
\item $Z_i$ and $Z_j$ agree over $D_{ij}$, or
\item $Z_i$ and $Z_j$ are disjoint over $D_{ij}^0$, in which case the locus in $D_{ij}$ where $Z_i$ and $Z_j$ agree is either empty or it is a union of subsets $D_{ijk}$ for some indices $k$ 
such that $$M\le k< i.$$ 
Moreover, for such an index $k$, $Z_k$ agrees with $Z_i$ over $D_{ik}$, $Z_k$ agrees with $Z_j$ over $D_{jk}$, and, for any $z\in s_k(D_{ijk}^0)$, 
\begin{equation}\label{svsdasdvsdv}
T_{z, s_i(D_{ij})}\cap T_{z, s_j(D_{ij})}=T_{z, s_k(D_{ijk})}.
\cooltag\end{equation} 
\end{enumerate}
\item If $i, j, k\geq M$ are such that $D_{ijk}\ne\emptyset$, then there exists a subset $\{a,b\}$ of $\{i,j,k\}$
such that $Z_a$ and $Z_b$ agree over $D_{ab}$.
\end{enumerate}
\end{definition}

\begin{remarks}
(a) Def. \ref{compatible} gives sufficient conditions to iterate elementary transformations along a sequence of data (see Prop. \ref{mvcvcvcnb} - the role of $M$ being to help formulate the inductive step).  Note that a compatible sequence of sections starting at $M$, with respect to $D_1,\ldots, D_N$, is the same as a compatible sequence of sections starting at $1$, with respect to $D_M,\ldots, D_N$, appropriately reindexed (i.e., we ignore $D_1,\ldots, D_{M-1}$).  

(b) In general, when making an elementary transformation along $(D, Z)$, the proper transform of a section may not be a section. By Lemma \ref{aux}, this holds, however, when the section either agrees with $Z$, or is disjoint from it. Condition (1) in Def. \ref{compatible} guarantees that in a compatible sequence $Z_M,\ldots Z_N$, any $Z_i$ for $i>M$ either agrees with $Z_M$, or is disjoint from it. Hence, after the elementary transformation given by $(D_M, Z_M)$, the proper transform of $Z_i$ is still a section. 

(c) If  $Z_i$ and $Z_j$ are disjoint over $D_{ij}^0$, then $Z_i$ and $Z_j$ give distinct sections of the $\bP^1$-bundle $p^{-1}(D_{ij})\rightarrow D_{ij}$ and, hence, their intersection has pure codimension $4$ in $Y$, i.e., the locus $G$ where $Z_i$ and $Z_j$ agree, has pure codimension $3$ in $X$.  Moreover, as $G\subseteq D_{ij}\setminus D_{ij}^0=\cup_k D_{ijk}$, it follows that $G$ is a union of subsets $D_{ijk}$. 
Hence, condition (1)(b) simply states that one cannot have $k<M$ or $k\geq i$. 

(d) Condition (2) in Def. \ref{compatible} guarantees that in a compatible sequence $Z_M,\ldots Z_N$, if $j, i>M$, then either $Z_i$ and $Z_j$ agree over $D_{ij}$ (hence, after the elementary transformation given by $(D_M, Z_M)$, the proper transforms $Z'_i$ and $Z'_j$ still agree over $D_{ij}$) or, if not, then $Z_i'$ and $Z_j'$ become disjoint over $D_{Mij}^0$ (see the proof of Prop.  \ref{mvcvcvcnb}). 
\end{remarks}

\begin{proposition}\label{mvcvcvcnb}
Given a compatible sequence of sections $Z_M,\ldots, Z_N$ starting at $M$, let $p':\,Y'\to X$ be an elementary transformation given by the data $(D_M,Z_M)$. Let $Z_{M+1}',\ldots,Z_N'\subset Y'$ be the proper transforms of $Z_{M+1},\ldots,Z_N$. Then $Z_{M+1}',\ldots,Z_N'$ are sections of $p'$ which form a compatible sequence of sections starting at $M+1$. 

In~particular, given a 
compatible sequence of sections $Z_1,\ldots, Z_N$ starting at $1$, we can iterate elementary transformations (along the data $(D_i, Z_i)$), to get a sequence of $\bP^1$-bundles
$Y_0=Y$, $Y_1=Y'$, $\ldots$, $Y_N$ over $X$.
\end{proposition}

\begin{proof}[Proof of Prop. \ref{mvcvcvcnb}]
We first show that each $Z_i'$ is a section for each $i>M$.
By Lemma \ref{aux}, it suffices to show that $Z_M$ and $Z_i$ are either disjoint or agree over $D_{Mi}$. Suppose they do not agree over $D_{Mi}$ and are not disjoint. Then we are in situation (b) of condition (1) in Def.  \ref{compatible}. Since there are no indices $k$ such that $M\leq k< M$, it follows that the locus where $Z_i$ and $Z_M$ agree is empty; hence, we have a contradiction. 


Next we show that $Z_{M+1}',\ldots, Z_N'$ form a compatible sequence of sections starting at $M+1$. Notice that condition (2) is obvious because the elementary transformation is an isomorphism outside of $D_M$ (if $Z_a$ and $Z_b$ agree over $D_{ab}$, then $Z'_a$ and $Z'_b$ agree over $D_{ab}$ as well). So we only need to check condition (1). Take $M<i<j$ such that $D_{ij}\ne\emptyset$. As before, if $Z_i$ and $Z_j$ agree over $D_{ij}$, then $Z_i'$ and $Z_j'$ agree over $D_{ij}$ as well. If $Z_i$ and $Z_j$  do not agree, then let 
$$\cK:=\{k\in\{1,\ldots, N\}\,|\,\hbox{\rm $Z_i$ and $Z_j$ agree over $D_{ijk}$}\},$$
$$\cK':=\{k\in\{1,\ldots, N\}\,|\,\hbox{\rm $Z_i'$ and $Z_j'$ agree over $D_{ijk}$}\}.$$

It is clear that $\cK'\setminus\{M\}=\cK\setminus\{M\}$ and \eqref{svsdasdvsdv} is satisfied
for these indices $k$ (because the elementary transformation is an isomorphism over $D^0_{ijk}$). So we only need to check that $M\not\in \cK'$, i.e., that $Z'_i$ and $Z'_j$ do not agree over $D_{Mij}$. We can assume that $D_{Mij}\ne\emptyset$, as otherwise there is nothing to prove. Consider two cases. Firstly, suppose $M\not\in \cK$. By condition (2) of Def. \ref{compatible}, we may assume without loss of generality that $Z_M$ and $Z_i$ agree over $D_{Mi}$. Then $Z_M$ and $Z_j$ do not agree over $D_{Mj}$ and therefore, must be disjoint as proved above. It follows by Lemma \ref{aux}(a) (applied to $Z_i$ and $Z_j$ over $D_{ij}$) that $Z'_i$ and $Z'_j$ are disjoint over $D_{Mij}$. Secondly, suppose $M\in \cK$. Then by Lemma \ref{aux}(b) applied to $Z_i$ and $Z_j$ over $D_{ij}$, we have that $Z'_i$, $Z'_j$ are disjoint over $D_{Mij}^0$ and hence, $M\not\in \cK'$.
\end{proof}

Before we give the proof of Theorem \ref{asdazxvsfvsfvsdaqwf}, we recall some basic properties of birational contractions. Recall that a birational map $f: ~Y\dra~ X$ between smooth, projective varieties is called a \emph{birational contraction} if the inverse map $f^{-1}$ does not contract any divisor. Equivalently, given a common resolution $(p,q): W\ra Y\times X$, any $p$-exceptional divisor is $q$-exceptional \cite{HK}[Def. 1.0].  For such a $W$, we have $\rho(W)=\rho(X)+r$, where $r$ is the number of $p$-exceptional divisors. Note that if $f$ does not contract a divisor $D$ in $Y$, then $f$ is a local isomorphism at the generic point of $D$. Hence,  a birational contraction $f: Y\dra X$ is a small modification if and only if $f$ does not contract any divisor, or, equivalently, $\rho(X)=\rho(Y)$. 

\begin{lemma}\label{preimage}
Let $f: Y\ra X$ be a proper birational morphism of smooth varieties. Assume that $T\subset X$ is a smooth, irreducible closed subvariety with smooth, irreducible scheme-theoretic preimage $Z\subset Y$.  Consider the blow-ups 
$$\pi_1: \tilde{X}=\Bl_T(X)\ra X,\quad \pi_2: \tilde{Y}=\Bl_Z(Y)\ra Y$$ with exceptional divisors $E_T$ and $E_Z$. Then there is an induced birational proper morphism 
$\tilde{f}:\tilde{Y}\ra\tilde{X}$, such that $\tilde{f}(E_Z)=E_T$. 
\end{lemma}

\begin{proof}
By the universal property of blow-ups, there is a morphism $\tilde{f}$ such that $\pi_1\circ\tilde{f}=f\circ\pi_2$ and we have $\tilde{f}^{-1}(E_T)=E_Z$. It follows that $\tilde{f}$ is proper
and $\tilde{f}(E_Z)=E_T$. 
\end{proof}

\begin{lemma}\label{contract}
Assume that $f: Y\dra X$ is a birational map between normal, projective varieties and $\pi:\tilde{Y}\ra Y$ is the blow-up of a closed subvariety $Z\subseteq Y$, with exceptional divisor $E$. Assume that $$f\circ\pi:\tilde{Y}\dra X$$ contracts all the components of $E$. Then if $f\circ\pi$ is a birational contraction, then $f$ is a birational contraction. 
\end{lemma}

\begin{proof}
If $(p,q): W\ra\tilde{Y}\times X$ is a common resolution and any $p$-exceptional divisor is $q$-exceptional, then, clearly any $\pi\circ p$-exceptional divisor (i.e., $p$-exceptional or a proper transforms of a component of $E$) is $q$-exceptional. 
\end{proof}

\begin{proof}[Proof of Thm. \ref{asdazxvsfvsfvsdaqwf}.]
Choose general points $q_1,\ldots,q_n\in\bP^{n-2}$ and let 
$\pi:\,\Bl_{q_n}\bP^{n-2}\to\bP^{n-3}$
be a resolution of the linear projection away from $q_n$. Then $\pi$ is a $\bP^1$-bundle.
Let $p_i=\pi(q_i)$ for $i=1,\ldots,n-1$.  For any subset $I$ of $\{1,\ldots,n-1\}$ such that $1\le |I|\le n-4$, let $L_I\subset\bP^{n-3}$ be the linear subspace spanned by $p_i$ for $i\in I$. Notice that we have sections  $t_I:\,L_I\to \pi^{-1}(L_I)$ that send $L_I$ to the proper transform of the linear subspace in $\bP^{n-2}$ spanned by $q_i$, for $i\in I$. 
Let $\Psi:\,\oM_{0,n}\to\bP^{n-3}$ be the Kapranov map such that $\Psi(\delta_{I\cup\{n\}})=L_I$  for any subset $I$ as above \cite{Kapranov}. Let $\pi_0:\,Y\to \oM_{0,n}$ be the pull-back of $\pi$ and let $s_I:\,\delta_{I\cup\{n\}}\to \pi^{-1}(\delta_{I\cup\{n\}})$
be the pull-back of $t_I$ for each subset $I$ as above.
We order the boundary divisors $\delta_{I\cup\{n\}}$ according to $|I|$ (in increasing order)
and arbitrarily for fixed $|I|$. This gives an order - which we denote by $\prec$ - on the subsets $I$.

\begin{claim}\label{check compatible}
The sections $s_I$ form a compatible sequence of sections.
\end{claim}

Assuming Claim \ref{check compatible}, we prove that by Prop.~\ref{mvcvcvcnb}, the last elementary transformation $Y_N$  is a SQM of the blow-up of $\bP^{n-2}$ along the points $q_1,\ldots, q_n$ and the proper transforms of the linear subspaces spanned by $\{q_i\}_{i\in I}$ for all subsets $I\subset\{1,\ldots, n-1\}$ with $\leq n-4$ elements. Moreover, we prove that the required small modification $\tLM_{n+1}$ is the blow-up of $Y_N$ in the proper transforms of the linear subspaces spanned by $\{q_i\}_{i\in I}$ for all subsets $I$ with $n-3$ elements. 

Consider the successive blow-ups $$X_0=\Bl_{q_n}\bP^{n-2}, X_1,\ldots, X_N$$ of $X_0$ along the (proper transforms of the) linear subspaces $t_I(L_I)$ in $\bP^{n-2}$ spanned by $q_i$, for $i\in I$, with the subsets $I$ ordered as above ($|I|\leq n-4$). For each  $\bP^1$-bundle in the sequence $$Y_0=Y, Y_1,\ldots, Y_N,$$ consider the induced birational map 
$f_k: Y_k\dra X_k$. For example, $f_0: Y_0\ra X_0$ is the birational proper map $Y\ra\Bl_{q_n}\bP^{n-2}$. 
\begin{claim}\label{contractions}
The map $f_k: Y_k\dra X_k$ is a birational contraction for all $k$.  
\end{claim}

\begin{proof}
We do an induction on $k$. Clearly, the statement holds for $k=0$ as $f_0$ is a birational morphism between smooth projective varieties.

For each $I\subset\{1,\ldots, n-1\}$ ($|I|\leq n-4$), we let $U_I\subseteq\bP^{n-3}$ be the complement of all the subspaces $L_{I'}$ for all subsets $I'\prec I$ ($I'\neq I$). The order $\prec$ is such that $L_{I'}\subseteq L_I$ only if $I'\prec I$ (since  $L_{I'}\subseteq L_I$ if and only if $I'\subseteq I$). In particular, $L_I\cap U_I\neq\emptyset$ and $U_I\subseteq U_{I'}$ if $I'\prec I$.

We introduce some notation: for an open set $U\subseteq\bP^{n-3}$ and a map $f: W\ra\bP^{n-3}$ we denote $W_U=f^{-1}(U)$. We will use this for the 
$\bP^1$-bundles $\pi_i:~Y_i\ra\oM_{0,n}$ (via the Kapranov map $\Psi: \oM_{0,n}\ra\bP^{n-3}$) and the blow-ups $X_i$ of $X_0$ (via $\pi: X_0\ra\bP^{n-3}$).

Assume now that $k\geq1$ and $Y_k$ is the elementary transformation of $Y_{k-1}$ along $(\delta_{I\cup\{n\}}, s_I)$, for a fixed subset $I$ with $|I|\leq n-4$. If $$\tilde{Y}_k\ra Y_{k-1}$$ is the blow-up along the proper transform of  $s_I(\delta_{I\cup\{n\}})$, then $Y_k$ is the blow-down of $\tilde{Y}_k$ along the proper transform of $\pi_i^{-1}(\delta_{I\cup\{n\}})$. Recall that 
$$X_k\ra X_{k-1}$$ is the blow-up along the proper transform of $t_I(L_I)$.  By induction, the map $f_{k-1}$ is a birational contraction. To prove that $f_k$ is a birational contraction, using Lemma \ref{contract}, it is enough to prove that:
\begin{itemize}
\item[(1) ] $\tilde{Y}_k\dra X_k$ is a birational contraction;
\item[(2) ] $\pi_i^{-1}(\delta_{I\cup\{n\}})$ is contracted by $f_{k-1}$. 
\end{itemize}

Clearly, it is enough to check (1) and (2) over open sets that intersect the above divisors.  Note that for $I'\prec I$, the elementary transformation with center  $(\delta_{I'\cup\{n\}}, s_{I'})$ is an isomorphism away from $\delta_{I'\cup\{n\}}=\Psi^{-1}(L_{I'})$. Hence, for $0\leq i\leq k-1$, the bundles $Y_i$ are isomorphic over $U_I$, i.e.,  $(Y_i)_{U_I}\cong Y_{U_I}$. Similarly, the blow-ups $X_0, X_1,\ldots, X_{k-1}$ are also isomorphic  over $U_I$, since at each step we blow-up a subvariety whose image under $\pi$ lies in $L_{I'}$, for some $I'\prec I$. In particular, the induced birational morphism $$(f_{k-1})_{U_I}: (Y_{k-1})_{U_I}\ra (X_{k-1})_{U_I}$$ is proper (being the same as the map $Y_0\ra X_0$ over $U_I$). Moreover, as the section $s_I$ is (by definition) the pull-back of the section $t_I$, the same is true when we consider these sections restricted to $U_I$. If we let 
$$(t_I)_{U_I}:=t_I(L_I)\cap\pi^{-1}(U_I),\quad (Z_I)_{U_I}:=s_I(\Psi^{-1}(U_I)\cap\delta_{I\cup\{n\}}),$$
then the pull-back under $(f_{k-1})_{U_I}$ of $(t_I)_{U_I}$ is $(Z_I)_{U_I}$. Moreover, $(X_k)_{U_I}$ is the blow-up of $(X_{k-1})_{U_I}$ along $(t_I)_{U_I}$ and $(Y_k)_{U_I}$ is the elementary transformation of $(Y_{k-1})_{U_I}$ along $(Z_I)_{U_I}$: $(\tilde{Y}_k)_{U_I}$ is the blow-up of $(Y_{k-1})_{U_I}$ along $(Z_I)_{U_I}$, and $(Y_k)_{U_I}$ is the blow-down of $(\tilde{Y}_k)_{U_I}$ along the proper transform of $\pi_{k-1}^{-1}(\delta_{I\cup\{n\}}\cap\Psi^{-1}(U_I))$. We now check (1) and (2) over $U_I$ (which intersects $L_I$, over which all the blown-up or blown-down loci lie). Property (2) follows immediately, as
$$\pi_{k-1}^{-1}(\delta_{I\cup\{n\}}\cap\Psi^{-1}(U_I))=\pi_0^{-1}(\delta_{I\cup\{n\}}\cap\Psi^{-1}(U_I))$$
is mapped by $f_0$ (hence, $f_{k-1}$) to $\pi^{-1}(L_I\cap U_I)$. We apply Lemma \ref{preimage} to the morphism $(f_{k-1})_{U_I}: (Y_{k-1})_{U_I}\ra (X_{k-1})_{U_I}$ and closed subschemes $(t_I)_{U_I}$, $(Z_I)_{U_I}$ (both sections of $\bP^1$-bundles over a smooth base, with $(Z_I)_{U_I}$ the scheme theoretic preimage of $(t_I)_{U_I}$). It follows by Lemma \ref{preimage} that the birational map $(\tilde{Y}_k)_{U_I}\dra X_k$ is a birational contraction, as it is a local isomorphism at the generic points of the corresponding exceptional divisors. Hence, property (1) holds.
\end{proof}

As after each elementary transformation, the Picard number $\rho(Y_i)$ stays constant, while $\rho(X_i)$ increases by one after each blow-up, it follows that $\rho(Y_N)=\rho(Y)=\rho(\oM_{0,n})+1$ equals $\rho(X_N)$. Hence, using Claim \ref{contractions}, it follows that the induced birational map $f_N: Y_N\dra X_N$ is a small modification. As in the proof of Claim \ref{contractions}, for all $I\subset\{1,\ldots, n-2\}$  such that $|I|=n-3$, the proper transform in $X_N$ of the subspace spanned by $\{q_i\}_{i\in I}$ does not lie in the indeterminacy locus of $f_N$. Moreover,  blowing up successively these loci and their proper transforms in $Y_N$ leads to a sequence of small modifications $f_{N+1}, f_{N+2},\ldots$, the last of which gives the required small modification $\tLM_{n+1}$.

\begin{proof}[Proof of Claim \ref{check compatible}]
Set $D_I:=\delta_{I\cup\{n\}}$.
Suppose $I\ne J$, $|I|\le |J|$, $D_{IJ}\ne\emptyset$. Then either $I\subset J$, in which case $Z_I$ and $Z_J$ agree over $D_{IJ}$, or there exists a partition $A\sqcup B\sqcup C=\{1,\ldots,n-1\}$ such that $I=A\cup B$ and $J=A\cup C$. In this case, the set $\cK$ from condition (1) of the compatible sequence is the set of all non-empty subsets of $A$.
This shows condition (2) and all of condition (1), except \eqref{svsdasdvsdv}.  If $A=\emptyset$ then there is nothing to check. Assume $A\neq\emptyset$. Let $\alpha\in D_{KIJ}^0$. It is enough to find a curve $C$ in $D_{IJ}$ passing through $\alpha$, such that in the ruled surface $S:=p^{-1}(C)$, $s_I$ and $s_J$ are not tangent above $\alpha$. As we have
$$\Psi(D_{IJ})=L_I\cap L_J\cong\bP^{|A|}, \quad \Psi(D_{IJK})=L_K\subseteq
 L_A\cong\bP^{|A|-1},$$
we may choose $l$ to be any line in $L_I\cap L_J$ that passes through $\Psi(\alpha)$ and is not contained in $L_A$. Let $C$ be any curve in $D_{IJ}$ that maps to $l$ and is smooth at $\alpha$. We claim that $C$ has the desired property, i.e., that $s_I(C)$ and  $s_J(C)$ are not tangent above $\alpha$. It suffices to check this after composing with the map $\Psi':\,Y\to \Bl_{q_n}\bP^{n-2}$, the pull-back of the Kapranov map, and the blow-up map $\Bl_{q_n}\bP^{n-2}\to \bP^{n-2}$. Let $\Lambda$ be the plane in $\bP^{n-2}$ which is the image of $p^{-1}(l)$. If $Z_I$ is the linear subspace in $\bP^{n-2}$ spanned by the points $q_i$ for $i\in I$, then $Z_I\cap Z_J=Z_A$. Clearly, the linear subspaces $Z_I\cap\Lambda$ and $Z_J\cap\Lambda$ intersect only at a point (lying above $L_A\cap l=\Psi(\alpha)$). Equivalently, $Z_I\cap\Lambda$ and $Z_J\cap\Lambda$ are not tangent at their intersection point. This 
proves the claim. 
\end{proof}

\end{proof}



The proof of Thm. \ref{asdazxvsfvsfvsdaqwf} and Cor. \ref{main} yield the following:
\begin{corollary}\label{codim 3}
Let $p_1,\ldots, p_{n-2}\in\bP^{n-3}$ be points in linearly general position and let $X_n$ be the toric variety which is the blow-up of $\bP^{n-3}$ along the proper transforms of linear subspaces of codimension $\geq 3$ spanned by the points $p_i$, in order of increasing dimension. Let $e$ denote the identity of the open torus of $X_n$. Then $\Bl_eX_{n+1}$ is a SQM of a $\bP^1$-bundle over $\oM_{0,n}$. If  $\ch k=0$ and $n\geq 134$, then $\Bl_eX_{n+1}$ is not a MDS. 
\end{corollary}




\section*{References}


\begin{biblist}




\bib{Artin}{article}{
    AUTHOR = {Artin, Michael},
     TITLE = {Some numerical criteria for contractability of curves on
              algebraic surfaces},
   JOURNAL = {Amer. J. Math.},
  FJOURNAL = {American Journal of Mathematics},
    VOLUME = {84},
      YEAR = {1962},
     PAGES = {485--496},
}


\bib{AGS}{article}{
   author={Alexeev, Valery},
   author={Gibney, Angela}, 
   author={Swinarski, David},  
   title={Conformal blocks divisors on $\bar{M}_{0,n}$ from $sl_2$},
   eprint={arXiv:1011.6659v1},
   date={2010},
}


\bib{BCHM}{article}{
AUTHOR = {Birkar, Caucher},
AUTHOR = {Cascini, Paolo},
AUTHOR = {Hacon, Christopher D.},
AUTHOR = {M\textsuperscript{c}Kernan, James},
     TITLE = {Existence of minimal models for varieties of log general type},
   JOURNAL = {J. Amer. Math. Soc.},
  FJOURNAL = {Journal of the American Mathematical Society},
    VOLUME = {23},
      YEAR = {2010},
    NUMBER = {2},
     PAGES = {405--468},
       URL = {http://dx.doi.org/10.1090/S0894-0347-09-00649-3},
}


\bib{BGM}{article}{
  AUTHOR = {Belkale, Prakash},
  AUTHOR = {Gibney, Angela},
  AUTHOR = {Mukhopadhyay, Swarnava},
   title={Quantum cohomology and conformal blocks on $\oM_{0,n}$},
   eprint={arXiv:1308.4906},
   date={2013},
}


\bib{Hausen}{article}{
  AUTHOR = {B\"aker, Hendrik},
  AUTHOR = {Hausen, Juergen},
  AUTHOR = {Keicher, Simon},
  title={On Chow quotients of torus actions},
   eprint={arXiv:1203.3759},
   date={2012},
}


\bib{BP}{incollection}{
AUTHOR = {Batyrev, Victor V.}
AUTHOR= {Popov, Oleg N.},
     TITLE = {The {C}ox ring of a del {P}ezzo surface},
 BOOKTITLE = {Arithmetic of higher-dimensional algebraic varieties ({P}alo
              {A}lto, {CA}, 2002)},
    SERIES = {Progr. Math.},
    VOLUME = {226},
     PAGES = {85--103},
 PUBLISHER = {Birkh\"auser Boston},
   ADDRESS = {Boston, MA},
      YEAR = {2004},
}
		

\bib{C}{article}{
AUTHOR = {Castravet, Ana-Maria},
     TITLE = {The {C}ox ring of {$\overline M_{0,6}$}},
   JOURNAL = {Trans. Amer. Math. Soc.},
  FJOURNAL = {Transactions of the American Mathematical Society},
    VOLUME = {361},
      YEAR = {2009},
    NUMBER = {7},
     PAGES = {3851--3878},
       URL = {http://dx.doi.org/10.1090/S0002-9947-09-04641-8},
}



\bib{CC}{article}{
   author={Coskun, Izzet},
   author={Chen, Dawei},  
   title={Extremal effective divisors on the moduli space of n-pointed genus one curves},  
   eprint={arXiv:1304.0350},
   date={2013},
}




\bib{CLS}{book}{
AUTHOR = {Cox, David A.}, 
AUTHOR= {Little, John B.},
AUTHOR={Schenck, Henry K.},
     TITLE = {Toric varieties},
    SERIES = {Graduate Studies in Mathematics},
    VOLUME = {124},
 PUBLISHER = {American Mathematical Society},
   ADDRESS = {Providence, RI},
      YEAR = {2011},
}


%
%


\bib{CT1}{article}{
AUTHOR = {Castravet, Ana-Maria}
AUTHOR= {Tevelev, Jenia},
     TITLE = {Hypertrees, projections, and moduli of stable rational curves},
   JOURNAL = {J. Reine Angew. Math.},
  FJOURNAL = {Journal f\"ur die Reine und Angewandte Mathematik. [Crelle's
              Journal]},
    VOLUME = {675},
      YEAR = {2013},
     PAGES = {121--180},
}


\bib{CT2}{article}{
AUTHOR = {Castravet, Ana-Maria}
AUTHOR= {Tevelev, Jenia},
 TITLE = {Rigid curves on {$\overline M_{0,n}$} and arithmetic
              breaks},
 BOOKTITLE = {Compact moduli spaces and vector bundles},
    SERIES = {Contemp. Math.},
    VOLUME = {564},
     PAGES = {19--67},
 PUBLISHER = {Amer. Math. Soc.},
   ADDRESS = {Providence, RI},
      YEAR = {2012},
}


\bib{Cutkosky}{article}{
AUTHOR = {Cutkosky, Steven Dale},
     TITLE = {Symbolic algebras of monomial primes},
   JOURNAL = {J. Reine Angew. Math.},
  FJOURNAL = {Journal f\"ur die Reine und Angewandte Mathematik},
    VOLUME = {416},
      YEAR = {1991},
     PAGES = {71--89},
      ISSN = {0075-4102},
}


\bib{Eisenbud}{book}{
AUTHOR = {Eisenbud, David},
     TITLE = {Commutative algebra},
    SERIES = {Graduate Texts in Mathematics},
    VOLUME = {150},
      NOTE = {With a view toward algebraic geometry},
 PUBLISHER = {Springer-Verlag},
   ADDRESS = {New York},
      YEAR = {1995},
     PAGES = {xvi+785},
}


\bib{Fed}{article}{
  AUTHOR = {Fedorchuk, Maksym},
  title={Cyclic covering morphisms on $\oM_{0,n}$},
   eprint={arXiv:1105.0655},
   date={2011},
}


\bib{Fulton}{book}{
    AUTHOR = {Fulton, William},
     TITLE = {Algebraic curves},
    SERIES = {Advanced Book Classics},
 PUBLISHER = {Addison-Wesley Publishing Company Advanced Book Program},
   ADDRESS = {Redwood City, CA},
      YEAR = {1989},
}


\bib{GianGib}{article}{
  AUTHOR = {Giansiracusa, Noah},
 AUTHOR = {Gibney, Angela},
     TITLE = {The cone of type {$A$}, level 1, conformal blocks divisors},
   JOURNAL = {Adv. Math.},
  FJOURNAL = {Advances in Mathematics},
    VOLUME = {231},
      YEAR = {2012},
    NUMBER = {2},
     PAGES = {798--814},
}


\bib{Milena}{article}{
AUTHOR ={Gonzalez, Jose}, 
AUTHOR ={Hering, Milena}, 
AUTHOR = {Payne, Sam},
AUTHOR = {S\"uss, Hendrik}
     TITLE = {Cox rings and pseudoeffective cones of projectivized toric
              vector bundles},
   JOURNAL = {Algebra Number Theory},
  FJOURNAL = {Algebra \& Number Theory},
    VOLUME = {6},
      YEAR = {2012},
    NUMBER = {5},
     PAGES = {995--1017},
}


\bib{GianJenMoon}{article}{
  AUTHOR = {Giansiracusa, Noah},
AUTHOR = {Jensen, David},
AUTHOR={Moon, Han-Bom},
     TITLE = {G{IT} compactifications of {$M_{0,n}$} and flips},
   JOURNAL = {Adv. Math.},
  FJOURNAL = {Advances in Mathematics},
    VOLUME = {248},
      YEAR = {2013},
     PAGES = {242--278},
}


\bib{GKM}{article}{
AUTHOR = {Gibney, Angela}, 
AUTHOR = {Keel, Sean},
AUTHOR = {Morrison, Ian},
     TITLE = {Towards the ample cone of {$\overline M_{g,n}$}},
   JOURNAL = {J. Amer. Math. Soc.},
  FJOURNAL = {Journal of the American Mathematical Society},
    VOLUME = {15},
      YEAR = {2002},
    NUMBER = {2},
     PAGES = {273--294},
}


\bib{GM}{article}{
AUTHOR = {Gibney, Angela} 
AUTHOR = {Maclagan, Diane},
     TITLE = {Equations for {C}how and {H}ilbert quotients},
   JOURNAL = {Algebra Number Theory},
  FJOURNAL = {Algebra \& Number Theory},
    VOLUME = {4},
      YEAR = {2010},
    NUMBER = {7},
     PAGES = {855--885},
      ISSN = {1937-0652},
}


\bib{GM_nef}{article}{
AUTHOR = {Gibney, Angela} 
AUTHOR = {Maclagan, Diane},
TITLE = {Lower and upper bounds for nef cones},
   JOURNAL = {Int. Math. Res. Not. IMRN},
  FJOURNAL = {International Mathematics Research Notices. IMRN},
      YEAR = {2012},
    NUMBER = {14},
     PAGES = {3224--3255},
}


\bib{GN_ams}{book}{
AUTHOR = {Goto, Shiro},
AUTHOR = {Nishida, Koji},
     TITLE = {The {C}ohen-{M}acaulay and {G}orenstein {R}ees algebras
              associated to filtrations},
      NOTE = {Mem. Amer. Math. Soc. {{\bf{110}}} (1994), no. 526},
 PUBLISHER = {American Mathematical Society},
   ADDRESS = {Providence, RI},
      YEAR = {1994},
     PAGES = {i--viii and 1--134},
}


\bib{GNW}{article}{	
AUTHOR = {Goto, Shiro},
AUTHOR = {Nishida, Koji},
AUTHOR = {Watanabe, Keiichi},
     TITLE = {Non-{C}ohen-{M}acaulay symbolic blow-ups for space monomial
              curves and counterexamples to {C}owsik's question},
   JOURNAL = {Proc. Amer. Math. Soc.},
  FJOURNAL = {Proceedings of the American Mathematical Society},
    VOLUME = {120},
      YEAR = {1994},
    NUMBER = {2},
     PAGES = {383--392},
}
	

\bib{HK}{article}{
    AUTHOR = {Hu, Yi},
    AUTHOR = {Keel, Sean},
     TITLE = {Mori dream spaces and {GIT}},      
   JOURNAL = {Michigan Math. J.},
  FJOURNAL = {Michigan Mathematical Journal},
    VOLUME = {48},
      YEAR = {2000},
     PAGES = {331--348},
       URL = {http://dx.doi.org/10.1307/mmj/1030132722},
}


\bib{HM}{article}{
 AUTHOR = {Harris, Joe},
AUTHOR={Mumford, David},
     TITLE = {On the {K}odaira dimension of the moduli space of curves},
      NOTE = {With an appendix by William Fulton},
   JOURNAL = {Invent. Math.},
  FJOURNAL = {Inventiones Mathematicae},
    VOLUME = {67},
      YEAR = {1982},
    NUMBER = {1},
     PAGES = {23--88},
}


\bib{Huneke}{article}{
 AUTHOR = {Huneke, Craig},
     TITLE = {Hilbert functions and symbolic powers},
   JOURNAL = {Michigan Math. J.},
  FJOURNAL = {The Michigan Mathematical Journal},
    VOLUME = {34},
      YEAR = {1987},
    NUMBER = {2},
     PAGES = {293--318}, 
}
		

\bib{Kapranov}{article}{
AUTHOR = {Kapranov, M. M.},
     TITLE = {Veronese curves and {G}rothendieck-{K}nudsen moduli space
              {$\overline M_{0,n}$}},
   JOURNAL = {J. Algebraic Geom.},
  FJOURNAL = {Journal of Algebraic Geometry},
    VOLUME = {2},
      YEAR = {1993},
    NUMBER = {2},
     PAGES = {239--262},
}



\bib{Keel}{article}{
AUTHOR = {Keel, Se{\'a}n},
     TITLE = {Basepoint freeness for nef and big line bundles in positive
              characteristic},
   JOURNAL = {Ann. of Math. (2)},
  FJOURNAL = {Annals of Mathematics. Second Series},
    VOLUME = {149},
      YEAR = {1999},
    NUMBER = {1},
     PAGES = {253--286},
      ISSN = {0003-486X},
}


\bib{Kiem}{article}{
  AUTHOR = {Kiem, Young-Hoon},
  title={Curve counting and birational geometry of compactified moduli spaces of curves},
Journal = {Proceedings of the Waseda symposium on algebraic geometry},  
year ={2010},
eprint={http://www.math.snu.ac.kr/~kiem/recentpapers.html},
}


\bib{Kurano-Matsuoka}{article}{
AUTHOR = {Kurano, Kazuhiko}, 
AUTHOR = {Matsuoka, Naoyuki},
     TITLE = {On finite generation of symbolic {R}ees rings of space
              monomial curves and existence of negative curves},
   JOURNAL = {J. Algebra},
  FJOURNAL = {Journal of Algebra},
    VOLUME = {322},
      YEAR = {2009},
    NUMBER = {9},
     PAGES = {3268--3290},
}


\bib{KM}{article}{
   AUTHOR = {Keel, Sean},
   AUTHOR = {M\textsuperscript{c}Kernan, James},  
   title={Contractible Extremal Rays on $\overline{M}_{0,n}$},
   eprint={arXiv:alg-geom/9607009v1}
   date={1996},
}

	
\bib{Larsen}{article}{
AUTHOR = {Larsen, Paul},
     TITLE = {Permutohedral spaces and the {C}ox ring of the moduli space of
              stable pointed rational curves},
   JOURNAL = {Geom. Dedicata},
  FJOURNAL = {Geometriae Dedicata},
    VOLUME = {162},
      YEAR = {2013},
     PAGES = {305--323},
      ISSN = {0046-5755},
}


\bib{Laz}{book}{ 
    AUTHOR = {Lazarsfeld, Robert},
     TITLE = {Positivity in algebraic geometry. {I}},
    SERIES = {Ergebnisse der Mathematik und ihrer Grenzgebiete. 3. Folge. A
              Series of Modern Surveys in Mathematics},
    VOLUME = {48},
 PUBLISHER = {Springer-Verlag},
   ADDRESS = {Berlin},
      YEAR = {2004},
}

\bib{LM}{article}{
AUTHOR = {Losev, A.},
AUTHOR= {Manin, Y.},
     TITLE = {New moduli spaces of pointed curves and pencils of flat
              connections},
   JOURNAL = {Michigan Math. J.},
  FJOURNAL = {The Michigan Mathematical Journal},
    VOLUME = {48},
      YEAR = {2000},
     PAGES = {443--472}
}
	
\bib{Mar}{article}{
  AUTHOR = {Maruyama, Masaki},
     TITLE = {Elementary transformations in the theory of algebraic vector
              bundles},
 BOOKTITLE = {Algebraic geometry ({L}a {R}\'abida, 1981)},
    SERIES = {Lecture Notes in Math.},
    VOLUME = {961},
     PAGES = {241--266},
 PUBLISHER = {Springer},
   ADDRESS = {Berlin},
      YEAR = {1982},
}


\bib{McK_survey}{article}{
AUTHOR = {McKernan, James},  
     TITLE = {Mori dream spaces},
   JOURNAL = {Jpn. J. Math.},
  FJOURNAL = {Japanese Journal of Mathematics},
    VOLUME = {5},
      YEAR = {2010},
    NUMBER = {1},
     PAGES = {127--151},
       URL = {http://dx.doi.org.proxy.lib.ohio-state.edu/10.1007/s11537-010-0944-7},
}


\bib{Okawa}{article}{
   author={Okawa, Shinnosuke},
   title={On images of Mori dream spaces},
   eprint={arXiv:1104.1326}
   date={2011},
}


\end{biblist}
\end{document}